\documentclass[leqno,11pt]{amsart}
\usepackage{amsmath}
\usepackage{amssymb,amscd}
\usepackage{amsthm}
\usepackage{latexsym}
\usepackage[myheadings]{fullpage}
\usepackage{color}  

\newtheorem{theorem}{Theorem}[section]
\newtheorem{lemma}[theorem]{Lemma}
\newtheorem{proposition}[theorem]{Proposition}
\newtheorem{corollary}[theorem]{Corollary}
\newtheorem{remark}[theorem]{Remark}
\newtheorem{example}[theorem]{Example}
\newtheorem{question}[theorem]{Question}
\newtheorem{definition}[theorem]{Definition}

\newcommand{\f}{{\bf{f}}}

\def\opn#1#2{\def#1{\operatorname{#2}}} 

\opn\PF{PF}
\opn\F{F}
\opn\G{G}
\opn\RF{RF}

\title{Nearly Gorenstein vs almost Gorenstein affine monomial curves}

\author{Alessio Moscariello}
\address{A. Moscariello - Dipartimento di Matematica -  Universit\`a di Pisa - Largo Bruno Pontecorvo 5 - 56127 Pisa - Italy.}
\email{alemoscariello@hotmail.it}

\author{Francesco Strazzanti}
\address{F. Strazzanti - Dipartimento di Matematica e Informatica - Universit\`a degli Studi di Catania - Viale Andrea Doria 6 - 95125 Catania - Italy.}
\email{francesco.strazzanti@gmail.com}

\thanks{The second author was supported by INdAM, more precisely he was ``titolare di un Assegno di Ricerca dell'Istituto Nazionale di Alta Matematica''.}

\subjclass[2010]{13H10, 20M14, 20M25}

\begin{document}

\begin{abstract}
We extend some results on almost Gorenstein affine monomial curves to the nearly Gorenstein case. In particular, we prove that the Cohen-Macaulay type of a nearly Gorenstein monomial curve in $\mathbb{A}^4$ is at most $3$, answering a question of Stamate in this particular case. Moreover, we prove that, if $\mathcal C$ is a nearly Gorenstein affine monomial curve which is not Gorenstein and $n_1, \dots, n_{\nu}$ are the minimal generators of the associated numerical semigroup, the elements of $\{n_1, \dots, \widehat{n_i}, \dots, n_{\nu}\}$ are relatively coprime for every $i$.
\end{abstract}

\keywords{Nearly Gorenstein ring, almost Gorenstein ring, 4-generated numerical semigroup, type of a numerical semigroup}

\maketitle

\section*{Introduction}

Let $k$ be a field and let $R$ be a Cohen-Macaulay positively graded $k$-algebra with graded maximal ideal $\mathfrak m$. Assume that $R$ admits a canonical module $\omega_R$ and let 
\[{\rm tr}(\omega_R)=\sum_{\varphi \in {\rm Hom}_R(\omega_R,R)} \varphi(\omega_R)
\] 
be the trace ideal of $\omega_R$. If $\mathfrak p \in {\rm Spec}(R)$, in \cite[Lemma 2.1]{HHS} it is proved that the ring $R_{\mathfrak p}$ is not Gorenstein if and only if ${\rm tr}(\omega_R) \subseteq \mathfrak p$, thus ${\rm tr}(\omega_R)$ describes the non-Gorenstein locus of $R$. In particular, $R$ is Gorenstein if and only if ${\rm tr}(\omega_R)=R$. In \cite{HHS} Herzog, Hibi, and Stamate call a ring for which $\mathfrak m \subseteq {\rm tr}(\omega_R)$ nearly Gorenstein and provide an extensive study of these rings. Clearly these are Gorenstein on the punctured spectrum, but the converse is not true.

Another generalization of Gorenstein ring is given by the notion of almost Gorenstein ring, introduced by Barucci and Fr\"oberg \cite{BF} in the case of analytically unramified rings of dimension one and generalized in \cite{GMP} and \cite{GTT}. In general, nearly and almost Gorensteinness are two unrelated notions, but in dimension one an almost Gorenstein ring is always nearly Gorenstein. 

In \cite{BF} the definition of almost Gorenstein ring arises in the context of numerical semigroups, indeed the authors introduce first the similar notion of almost symmetric numerical semigroup. The aim of the present paper is to study the relations between almost symmetric and nearly Gorenstein numerical semigroup rings and, in particular, we extend some properties of almost symmetric semigroups to the nearly Gorenstein case. 

We recall that a numerical semigroup $S$ is a submonoid of the natural numbers $\mathbb{N}$ such that $\mathbb{N}\setminus S$ is finite, while the numerical semigroup ring associated with $S$ and a field $k$ is the one-dimensional domain $k[S]=k[t^s \mid s \in S]$. If $S$ is minimally generated by $n_1, \dots, n_\nu$, the ring $k[S]$ is isomorphic to the coordinate ring of the curve in $\mathbb{A}^{\nu}$ parametrized by the monomials $t^{n_1}, t^{n_2}, \dots, t^{n_{\nu}}$. Bearing in mind this bijection, studying the properties of an affine monomial curve is equivalent to study its associated numerical semigroup.

It is well-known that the Cohen-Macaulay type of a numerical semigroup ring does not exceed two, if it has embedding dimension at most three. This turns out to be false in embedding dimension $4$, in fact there is no upper bound for the Cohen-Macaulay type in this case, see \cite[Example p. 75]{FGH}. On the other hand, if $k[S]$ is almost Gorenstein and it has embedding dimension $4$, Numata \cite{Nu3} asked if the Cohen-Macaulay type of $k[S]$ is at most three; this was indeed proved by the first author in \cite{M} by using the new notion of row factorization matrix. In \cite{M} it is also asked if there exists an upper bound for the Cohen-Macaulay type of an almost Gorenstein numerical semigroup ring in terms of its embedding dimension. Generalizing this question, in \cite{S} Stamate raised the same problem for nearly Gorenstein rings. Here we give a positive answer in embedding dimension four by proving that also in this case the Cohen-Macaulay type is at most three. To achieve this result we give a new characterization of nearly Gorenstein numerical semigroups and we introduce a different kind of row factorization matrix which might be useful also in further studies of these semigroups. Moreover, we prove that a nearly Gorenstein numerical semigroup can be obtained by gluing only if it is symmetric and we characterize when a numerical semigroup generated by a generalized arithmetic sequence is nearly Gorenstein.

The structure of the paper is the following. In the first section we fix the notation and recall the basic definitions and results. In Proposition \ref{Nearly Gorenstein} we also prove a useful characterization of nearly Gorenstein numerical semigroups and we introduce the notion of NG-vector.
In the second section we study the type of a nearly Gorenstein numerical semigroup $S$ and the main result is Theorem \ref{t<4}, where we prove that the type of $S$ is at most three if $S$ has embedding dimension four. In Section 3 we study when gluing and generalizing arithmetic sequences are nearly Gorenstein and, as a consequence, we obtain that if $S=\langle n_1, \dots, n_\nu \rangle$ is a nearly Gorenstein semigroup with embedding dimension $\nu$ and $S$ is not symmetric, then the elements of $\{n_1, \dots, \widehat{n_i}, \dots, n_{\nu}\}$ are relatively coprime for every $i$, see Corollary \ref{coprime}. Finally in the last section we raise some questions for further developments.

The computations appearing in this paper were performed by using the GAP system \cite{GAP} and, in particular, the NumericalSgps package \cite{DGM}.

\section{Preliminaries and NG-vectors}

Let $S$ be a numerical semigroup, i.e. a finite submonoid of $\mathbb{N}$ such that $\mathbb{N}\setminus S$ is finite. 
Given $n_1, \dots, n_{\nu} \in \mathbb{N}$ with $\gcd(n_1,\dots,n_{\nu})$ we define the numerical semigroup $\langle n_1, \dots, n_{\nu}\rangle=\{\sum_{i=1}^{\nu} a_i n_i \mid a_i \in \mathbb{N} \text{ for every } i\}$ and we say that $\{n_1, \dots, n_{\nu}\}$ is a set of generators of this semigroup. 
It is well-known that every numerical semigroup has a unique set of minimal generators and it is finite. We denote it by $\G(S)$ and we refer to its cardinality as the embedding dimension of $S$.
The maximum of $\mathbb{N}\setminus S$ is called Frobenius number of $S$ and it is denoted by $\F(S)$. Let also $\PF(S)=\{f \in \mathbb{Z} \mid f+s \in S \text{ for any } s \in S\setminus \{0\}\}$ be the set of pseudo-Frobenius numbers of $S$ and note that $\F(S) \in \PF(S)$. The type $t(S)$ of $S$ is the cardinality of $\PF(S)$.
We note that the embedding dimension and the type of $S$ are equal to the embedding dimension and the Cohen-Macaulay type of $k[S]$ respectively, where $k$ is a field. 
In particular, $k[S]$ is Gorenstein if and only if $\PF(S)=\{\F(S)\}$ and in this case $S$ is said to be symmetric.

A relative ideal of $S$ is a set $I \subseteq \mathbb{Z}$ such that $I+S \subseteq I$ and there exists $x \in S$ for which $x+I \subseteq S$. Two important examples of relative ideals are the maximal ideal $M(S)=S \setminus \{0\}$ and the canonical ideal $K(S)=\{z \in \mathbb{Z} \mid \F(S)-z \notin S\} \supseteq S$. We recall that $K(S)=S$ if and only if $S$ is symmetric and that $K(S)$ is generated as relative ideal by the elements $\F(S)-f$ with $f \in \PF(S)$, i.e. $K(S)=\{\F(S)-f+s \mid f \in \PF(S), s \in S\}$. For more information about numerical semigroups we refer to \cite{BDF, RG1}.

A numerical semigroup $S$ is said to be almost symmetric if $M(S)+K(S) = M(S)$ and $k[S]$ is almost Gorenstein if and only if $S$ is almost symmetric. A nice characterization of these semigroups was proved by Nari in \cite[Theorem 2.4]{N}: $S$ is almost symmetric if and only if $\F(S)-f \in \PF(S)$ for all $f \in \PF(S)$.  

It follows by \cite[Lemma 1.1]{HHS} that $k[S]$ is nearly Gorenstein if and only if $M(S) \subseteq K(S)+(S-K(S))$ and in this case we simply say that $S$ is a nearly Gorenstein semigroup. It is known that an almost symmetric numerical semigroup is nearly Gorenstein because this implication holds for one-dimensional rings, see \cite[Proposition 6.1]{HHS}. Anyway it is possible to obtain this result as a consequence of the next proposition.

\begin{proposition} \label{Nearly Gorenstein}
The following statements hold:
\begin{enumerate}
\item $S$ is almost symmetric if and only if $n+\F(S)-f \in S$ for all $f \in \PF(S)$ and all $n \in \G(S)$;
\item $S$ is nearly Gorenstein if and only if for every $n_i \in \G(S)$ there exists $f_i \in \PF(S)$ such that $n_i+f_i-f \in S$ for all $f \in \PF(S)$. 
\end{enumerate}
In particular, an almost symmetric numerical semigroup is nearly Gorenstein.
\end{proposition}

\begin{proof}
(1) The second condition is equivalent to say that $\F(S)-f \in \PF(S)$ for every $f \in \PF(S)$. Therefore, the conclusion follows by Nari's characterization \cite[Theorem 2.4]{N}. \\
\noindent (2) We can assume that $S$ is not symmetric. Since the generators of $K(S)$ are $\F(S)-f$ with $f \in \PF(S)$, the second condition is equivalent to $x_i= n_i + f_i - \F(S) \in (S-K(S))$ for every $i=1, \dots, \nu$. If this holds, it is clear that $n_i=\F(S)-f_i +x_i \in K(S) + (S-K(S))$ for every minimal generator $n_i$ and, thus, $M(S) \subseteq K(S)+(S-K(S))$.

Conversely, assume that $S$ is nearly Gorenstein. Every generator of $S$ lies in $K(S)+(S-K(S))$ and $0 \notin (S-K(S))$, since $S$ is not symmetric. Let $n_i \in \G(S)$ and $n_i=k+x$ with $k \in K(S)$ and $x \in (S-K(S))$. Therefore, $k=\F(S)-f_i +s$ for some $f_i \in \PF(S)$ and some $s \in S$. Since $(\F(S)-f_i)+x \in S \setminus \{0\}$ and $n_i$ is a minimal generator, it follows that $s=0$ and $n_i+f_i - \F(S)=x \in (S-K(S))$ which yields the thesis.    
\end{proof}

\begin{definition} 
Let $S=\langle n_1, \dots, n_{\nu}\rangle$, where $n_1 < \dots < n_{\nu}$ are minimal generators. We call a vector $\f=(f_1,\dots, f_{\nu}) \in \PF(S)^{\nu}$ nearly Gorenstein vector for $S$, briefly {\rm NG}-vector, if $n_i+f_i-f \in S$ for every $f \in \PF(S)$ and every $i=1, \dots, \nu$.
\end{definition}

By Proposition \ref{Nearly Gorenstein} the existence of an NG-vector is equivalent to the nearly Gorensteinness of $S$, whereas $S$ is almost symmetric if and only if it admits the NG-vector $(\F(S), \dots, \F(S))$.

\begin{proposition} \label{minimum index}
Let $(f_1, \dots, f_{\nu})$ be an {\rm NG}-vector $S$. The following hold:
\begin{enumerate}
\item $f_1=\F(S)$;
\item If $i$ is the minimum index for which $f_i \neq \F(S)$, then $f_i=\F(S)-n_i+n_l$ for some $l < i$.
\end{enumerate}
\end{proposition}

\begin{proof}
(1) By definition of $f_1$ we have $n_1 + f_1 - \F(S) \in S$. If $n_1 + f_1 - \F(S)=0$, then $\F(S)=f_1+n_1 \in S$, which is a contradiction. Therefore, since $f_1 < \F(S)$ it follows that $n_1 + f_1 - \F(S) \leq n_1$ and, then, $n_1 + f_1 - \F(S)=n_1$, i.e. $f_1=\F(S)$. \\
(2) Since $n_i > n_i + f_i - \F(S) \in S$, it follows that $n_i+f_i-\F(S)=a_1 n_1 + \dots + a_{i-1}n_{i-1}$ for some non-negative integers $a_1, \dots, a_{i-1}$. At least one of these integers has to be positive, so assume that $a_l >0$; then, $n_i = \F(S)-f_i+n_l+ a_1n_1 + \dots (a_l-1)n_l + \dots a_{i-1}n_{i-1}$. From $\F(S)=f_l$, it follows that $\F(S)-f_i+n_l \in S \setminus \{0\}$ and, since $n_i$ is a minimal generator, this implies that $a_l=1$ and $a_j=0$ if $j \neq l$, i.e. $f_i=\F(S)-n_i+n_l$.
\end{proof}

We remark that $\f$ could be not unique. For instance the pseudo-Frobenius numbers of $S=\langle 4,5,11 \rangle$ are $\PF(S)=\{6,7\}$ and it is easy to see that $(7,6,6)$ and $(7,6,7)$ are the NG-vectors of $S$. In particular, in this case $S$ is nearly Gorenstein but not almost symmetric.

\section{On the type of a nearly Gorenstein semigroup}

Throughout this section we set $S=\langle n_1, \dots, n_\nu \rangle$, where $n_1 < n_2 < \dots < n_\nu$ are minimal generators. Moreover, if $S$ is nearly Gorenstein, we fix an NG-vector $\f=(f_1, \dots, f_{\nu})$.

For every $f \in \PF(S)$ and every $i=1, \dots, \nu$ we have $f+n_i=\sum_{j=1}^\nu a_{ij}n_j$ with $a_{ij}\geq 0$ and $a_{ii}=0$. A square matrix $A=(a_{ij})$ of order $\nu$ is said to be an RF-matrix (short for row-factorization matrix) for $f$ if $a_{ii}=-1$, $a_{ij} \in \mathbb{N}$ when $i \neq j$ and $f=\sum_{j=1}^\nu a_{ij}n_j$. In this paper we refer to this notion as  $\RF^+$ matrix to avoid confusion with another matrix that we are going to introduce.

If $S$ is nearly Gorenstein, for every $i$ such that $f \neq f_i$ we also have
\[
n_i+f_i-f=\sum_{j=1}^{\nu}b_{ij}n_j 
\]
with $b_{ij} \geq 0$ and $b_{ii}=0$; thus, we can define another matrix similarly to the previous case.

\begin{definition}
Let $(f_1, \dots, f_{\nu})$ be an {\rm NG}-vector for $S$ and let $f \in \PF(S)$. We say that a square matrix $B=(b_{ij})$ of order $\nu$ is an $\RF^-$ matrix for $f$ if B satisfies the following properties: if $f=f_i$ in the $i$-th row of $B$ there are only zeroes, otherwise $b_{ii}=-1$ and the entries $b_{ij}$ are such that $f_i-f=\sum_{j=1}^{\nu}b_{ij}n_j$.    
\end{definition}

Clearly this matrix depends on the NG-vector, but, even if we fix it, there could be more matrices associated with $f$.

\begin{example} \rm
The semigroup $S=\langle 10,12,37,75 \rangle$ is nearly Gorenstein because $(65,63,38,63)$ is an NG-vector for it. There is a unique $\RF^+$ matrices associated with $38 \in \PF(S)$, which is
\[
\begin{pmatrix}
-1&4&0&0 \\
5&-1&0&0 \\
0&0&-1&1 \\
4&3&1&-1 \\
\end{pmatrix},
\]
whereas there are two possible $\RF^-$ matrices for $38$:
\[
\begin{pmatrix}
-1&0&1&0 \\
0&-1&1&0 \\
0&0&0&0  \\
10&0&0&-1 \\
\end{pmatrix}, \ \ \  \ \ \ \ \ 
\begin{pmatrix}
-1&0&1&0 \\
0&-1&1&0 \\
0&0&0&0  \\
4&5&0&-1 \\
\end{pmatrix}.
\]
We note that also $(65,63,38,38)$ is an NG-vectors for $S$ and, if we choose it, there is only one $\RF^{-}$ matrix for $38$ because every entry in the last row has to be zero. 
\end{example}

\begin{lemma} \label{Coppie}
Let $S$ be nearly Gorenstein and let $f\in \PF(S)$. Also, let $A=(a_{ij})$ and $B=(b_{ij})$ be an $\RF^+$ and an $\RF^{-}$ matrix for $f$ respectively. Then, $a_{jk}b_{kj}=0$ for every $j\neq k$.  
\end{lemma}

\begin{proof}
Assume $f \neq f_k$, otherwise $b_{kj}=0$. Then, 
\[
f_k=f+(f_k-f)=(a_{j1}+b_{k1}) n_1 + \dots (a_{j\nu}+b_{k\nu}) n_\nu \notin S.
\]
Thus, at least one coefficient has to be negative and, since $a_{jj}=b_{kk}=-1$, the only possibilities are $a_{jk}-1=-1$ or $b_{kj}-1=-1$, that is $a_{jk}b_{kj}=0$.
\end{proof}

In \cite{M} it is proved that the type of an almost symmetric semigroup with four generators does not exceed three. In the following theorem we prove that this bound holds also for nearly Gorenstein semigroups.

\begin{theorem} \label{t<4}
If $S=\langle n_1, n_2, n_3, n_4 \rangle$ is nearly Gorenstein, then $t(S) \leq 3$. Moreover, if $S$ is not almost symmetric and $i$ is the minimum index such that $f_i \neq f_1$, then either
\begin{gather*}
\PF(S) = \{\F(S), \, \F(S)-n_i+n_l \} \text{ \ or } \\
\PF(S) = \{\F(S), \, \F(S)-n_i+n_l, \, \lambda n_k - n_j\},
\end{gather*}
where $l < i$, $\{i,j,k,l\}=\{1,2,3,4\}$ and $\lambda \in \mathbb{N}$.
\end{theorem}

\begin{proof}
We can assume that $S$ is not almost symmetric by \cite[Theorem 1]{M}. 
If $i$ is the minimum index such that $f_i \neq f_1$, then $f_i=f_1-n_i+n_l$ for some $l < i$ by Proposition \ref{minimum index}. Assume by contradiction that there exist three different integers $f,f',f'' \in \PF(S) \setminus \{f_1,f_i\}$ and let \begin{align*}
f_1-f&=-n_l+a_i n_i +a_j n_j + a_k n_k, \\
f_1-f'&=-n_l+b_i n_i +b_j n_j + b_k n_k, \\
f_1-f''&=-n_l+c_i n_i +c_j n_j + c_k n_k.
\end{align*}
Then, $f_i-f=f_1-f-n_i+n_l=(a_i-1)n_i+a_jn_j+a_kn_k \notin S$ and, thus, $a_i=0$; similarly $b_i=c_i=0$. Therefore, $f_1-f=-n_l+a_jn_j+a_kn_k$ and $f_i-f=-n_i+a_jn_j+a_kn_k$ and it is not possible that both $a_j$ and $a_k$ are zero. Since the same holds for $f'$ and $f''$, we can assume without loss of generality that $a_j \neq 0$ and $c_j \neq 0$. Then, there is an $\RF^-$ matrix for $f$ and $f''$ in which the $(i,j)$ and $(l,j)$ entries are positive and, in light of Lemma \ref{Coppie}, this means that the $(j,i)$ and $(j,l)$ entries of every $\RF^+$ matrix for $f$ and $f''$ are zero, i.e. $f=-n_j+\lambda n_k$ and $f''=-n_j+\gamma n_k$. Therefore, it follows that either $f - f'' \in S$ or $f''- f \in S$, which yields a contradiction. 

Hence, there are at most two pseudo-Frobenius numbers $f$ and $f'$ different from $f_1$ and $f_i$, moreover
\begin{align}
&f=-n_j+\lambda n_k   &f_1-f=-n_l + a_j n_j  \hspace{1cm}  &f_i-f=-n_i + a_j n_j\\
&f'=-n_k+\mu n_j  &f_1-f'=-n_l + b_k n_k   \hspace{1cm}  &f_i-f'=-n_i + b_k n_k. \label{f'}
\end{align}
By adding the first two equalities of every line, we get $f_1=-n_l+(a_j-1)n_j+\lambda n_k$ and $f_1=-n_l + (b_k-1)n_k+\mu n_j$, thus $(\mu+1-a_j)n_j=(\lambda+1-b_k)n_k$. Since $b_k$ is positive, $\lambda+1-b_k\leq \lambda$; moreover, if $\lambda+1-b_k >0$, then also $\mu+1-a_j >0$ and $f=\lambda n_k-n_j \geq_S (\lambda+1-b_k)n_k-n_j=(\mu-a_j)n_j \in S$ which yields a contradiction. Therefore, $\lambda+1-b_k\leq 0$, i.e. $b_k \geq \lambda+1$, and, since $(\lambda+1)n_k-n_j=f+n_k \in S$, it follows that $b_kn_k-n_j\in S$. 
This means that $b_k n_k-n_j=\alpha_ln_l+\alpha_i n_i+\alpha_j n_j+\alpha_k n_k$ with $\alpha_l,\alpha_i,\alpha_j,\alpha_k \geq 0$. If $\alpha_l$ (resp. $\alpha_i$) is positive, then by replacing $b_k$ in (\ref{f'}) we get $f_1-f' \in S$ (resp. $f_i-f' \in S$) which is a contradiction. It follows that
\begin{align*}
&f_1-f'=-n_l + (\alpha_j+1)n_j+\alpha_k n_k \\
&f_i-f'=-n_i + (\alpha_j+1)n_j+\alpha_k n_k.
\end{align*}
Hence, there is an $\RF^-$ matrix for $f'$ whose $(l,j)$ and $(i,j)$ entries are positive and, consequently, the $(j,l)$ and $(j,i)$ entries of every $\RF^+$ for $f'$ are zero, i.e. $f'=-n_j + \gamma n_k$. This implies that $f=f'$ and $t(S) \leq 3$.
\end{proof}

\begin{example} \rm
Consider the numerical semigroup $S=\langle 15,17,28,41 \rangle$ which is nearly Gorenstein with a unique NG-vector $(121,121,108,95)$. According to Proposition \ref{minimum index} we have $\F(S)=121$ and $108=\F(S)-28+15$. Moreover, the type of $S$ is three and the other pseudo-Frobenius number is $95=8*17-41$.
\end{example}

\begin{corollary} 
Let $S=\langle n_1, n_2, n_3, n_4 \rangle$ be nearly Gorenstein. 
\begin{enumerate}
\item Let $\F(S) \neq f_2=f_3=f_4$. Then, either $\PF(S)=\{f_2,\F(S)\}$ or $\PF(S)=\{f_2/2, f_2,\F(S)\}$.
\item Assume that $\F(S)=f_i=f_j \neq f_k$ with $\{i,j,k\}=\{2,3,4\}$. If there is $f \in \PF(S) \setminus \{f_k, \F(S)\}$, then either $f=\F(S)/2$ or $S$ is almost symmetric. 
\end{enumerate}
\end{corollary}

\begin{proof}
(1) Suppose that there exists $f \in \PF(S) \setminus \{f_2, \F(S)\}$. Since $\F(S)=f_1$, there is a factorization $\F(S)-f =-n_1+a_2 n_2+a_3n_3 + a_{4}n_4$ with $a_2, a_3, a_4 \geq 0$ and at least one of them positive; without loss of generality we may assume that $a_2 > 0$. Moreover, there exists a factorization $f_2 -\F(S) = b_1 n_1 -n_2 +  b_3 n_{3} +  b_4 n_4$ with $b_1, b_3, b_4 \geq 0$. Therefore, 
\[
f_2-f = (f_2-\F(S))+(\F(S)-f)=(b_1-1)n_1 +(a_2-1)n_2+(a_3+b_3)n_3+(a_4+b_4)n_4
\]
and, since $f_2 - f \notin S$ and $a_2 >0$, it follows that $b_1=0$ and $f_2-f+n_1 \in S$. Moreover, $f_2-f+n_i \in S$ for $i=2,3,4$ by hypothesis and, thus, $f_2-f \in \PF(S)$. Since $f_2-f < f_2 < \F(S)$, Theorem \ref{t<4} implies that $f_2-f=f$, i.e. $f=f_2/2$. \\
(2) By using the same argument of the previous case we can prove that $\F(S)-f \in \PF(S)$ and by Theorem \ref{t<4} it follows that either $\F(S)-f=f_k$ or $\F(S)-f=f$. In the first case Nari's Theorem \cite[Theorem 2.4]{N} implies that $S$ is almost symmetric, whereas in the second one $f=\F(S)/2$.  
\end{proof}

All the possibilities of the previous corollary may occur as the following examples show.

\begin{example} \rm
{\bf 1.} The semigroup $\langle 11,12,37,50 \rangle$ is nearly Gorenstein with $\f=(76,75,75,75)$ and it has type 2. \\
{\bf 2.} The semigroup $S=\langle 10,11,45,79 \rangle$ has the NG-vector $(69,68,68,68)$ and $\PF(S)=\{34,68,69\}$. We note that also $(69,68,68,34)$, $(69,68,68,69)$, $(69,68,34,68)$, $(69,68,34,34)$ and $(69,68,34,69)$ are NG-vectors for $S$. \\
{\bf 3.} The semigroup $\langle 10,11,12,19 \rangle$ is nearly Gorenstein with $\f=(37,37,37,28)$ and it has type $2$. \\
{\bf 4.} The semigroup $S=\langle 10,11,12,29 \rangle$ is nearly Gorenstein by choosing $\f=(38,38,37,38)$ and $\PF(S)=\{19,37,38\}$. Also in this case there are more NG-vectors: $(38,38,37,19)$, $(38,37,37,38)$ and $(38,37,37,19)$. \\
{\bf 5.} Let $S=\langle 8,9,11,15 \rangle$. An NG-vector for $S$ is $(21,21,21,14)$ and it is almost symmetric because $\PF(S)=\{7,14,21\}$. Indeed $(21,21,21,14)$ and $(21,21,21,21)$ are the only NG-vectors of $S$.
\end{example}

We conclude this section with some results about the general case.

\begin{proposition} \label{f1-fj}
Let $S$ be a nearly Gorenstein semigroup and suppose that $f_1, f_2, \dots, f_i$ are pairwise distinct for some $i \leq \nu$. The following statements hold:
\begin{enumerate}
	\item If $f \in \PF(S) \setminus \{f_1, \dots, f_i\}$, then $f_1-f=-n_1+a_{i+1}n_{i+1}+ \dots + a_{\nu}n_{\nu}$, with $a_j \geq 0$ for every $j=i+1, \dots, \nu$;
    \item $f_1-f_{j} =n_{j}-n_1$ for every $j=1, \dots i$.
\end{enumerate}
\end{proposition}

\begin{proof}
We proceed by induction on $i$. If $i=1$, (2) is trivial; moreover, if $f \neq f_1$, then $f_1-f \notin S$ and $f_1-f+n_1 \in S$, thus (1) follows.  

We assume that both statements hold for $i$ and we will prove them for $i+1$. We start proving (2). Since $f_{i+1} \notin \{f_1, \dots, f_i\}$, by induction we have $f_1 - f_{i+1}=-n_1 + a_{i+1}n_{i+1} + \dots + a_{\nu}n_\nu$. Moreover, $f_{i+1}-f_1 \notin S$ and $f_{i+1}-f_1+n_{i+1} \in S$, then $f_{i+1}-f_1 =c_1 n_1 + \dots + c_{i}n_i - n_{i+1} + \dots + c_\nu n_\nu$, with $c_j$ non-negative integer for every $j$. It follows that $n_1+n_{i+1}=c_1 n_1 + \dots + c_{i}n_i + a_{i+1}n_{i+1}+(a_{i+2} + c_{i+2})n_{i+2} + \dots (a_\nu + c_\nu)n_\nu$; clearly at least one $a_j$ and one $c_k$ are non-zero and, since $n_1 < \dots < n_\nu$, it follows that $c_1=a_{i+1}=1$ and the other coefficients are zero. Hence, $f_1-f_{i+1}=n_{i+1}-n_1$, which is (2).

Let now $f \in \PF(S) \setminus \{f_1, \dots, f_{i+1}\}$. By induction $f_1-f=-n_1+a_{i+1}n_{i+1}+ \dots + a_\nu n_\nu$ and, since $f_{i+1} - f = (f_1-f)+(f_{i+1}-f_1)$ using (2) it follows that $f_{i+1}-f=(a_{i+1}-1)n_{i+1}+ \dots + a_\nu n_\nu$. Therefore, $a_{i+1}=0$ because $f_{i+1}-f \notin S$ and $(1)$ follows.
\end{proof}

\begin{corollary} \label{Distinct}
If $S$ is nearly Gorenstein, there exist at least two different indices $i$ and $j$ such that $f_i=f_j$. Moreover, if $f_1, \dots, f_{\nu-1}$ are pairwise distinct, then $\PF(S)=\{f_1, \dots, f_{\nu-1}\}$ and $t(S)=\nu-1$.
\end{corollary}

\begin{proof}
Suppose by contradiction that there exists $f \in \PF(S) \setminus \{f_1, \dots, f_{\nu-1}\}$. The two statements of the previous proposition imply that $f_1-f=-n_1+a_{\nu}n_{\nu}$ with $a_{\nu} >0$ and $f_1-f_i=n_i-n_1$ for every $i< \nu$. Therefore,
$f_i-f=(f_1-f)-(f_1-f_i)=a_\nu n_\nu - n_i$ for every $i < \nu$. 
Consider a factorization $f+n_\nu = b_1 n_1 + \dots + b_{\nu-1}n_{\nu-1}$, where $b_i \geq 0$ and assume $b_k >0$ for a fixed $k$. Then 
$f_k=(f_k-f)+f= (a_\nu -1) n_\nu + b_1 n_1 + \dots + (b_{k}-1) n_k + \dots + b_{\nu-1} n_{\nu -1} \in S$, which is a contradiction. Hence, $\PF(S)=\{f_1, \dots, f_{\nu-1}\}$ and, in particular, it is not possible that $f_1, \dots, f_\nu$ are pairwise distinct.
\end{proof}

\begin{example} \rm
Consider $S=\langle 10, 11, 12, 14, 16, 29 \rangle$ that is nearly Gorestein with $f_1=19$, $f_2=18$, $f_3=17$, $f_4=15$, $f_5=13$ and $f_6 \in \{f_1,f_2,f_3,f_4,f_5\}$. According to the previous corollary $\PF(S)=\{f_1, f_2, f_3, f_4, f_5\}$ and $t(S)=5$.   
\end{example}

\begin{remark} \rm \label{type 9}
{\bf 1.} Despite Theorem \ref{t<4} and Corollary \ref{Distinct}, the type of a nearly Gorenstein semigroup can be greater than its embedding dimension. For instance the nearly Gorenstein numerical semigroup $S=\langle 64, 68, 73, 77, 84, 93 \rangle$ has embedding dimension $6$ and type $9$, since the pseudo-Frobenius numbers of $S$ are $\PF(S)=\{159, 179, 188, 195, 197, 206, 215, 394, 403 \}$. \\
{\bf 2.} Let $S$ be nearly Gorenstein. If either $S$ has embedding dimension four or $f_1, \dots, f_{\nu-1}$ are pairwise distinct, then $S$ satisfies Wilf's conjecture \cite{D}. Indeed $\F(S)+1 \leq n(S)(t(S)+1) \leq n(S)\nu(S)$, where the first inequality follows by \cite[Theorem 20]{FGH}.
\end{remark}

\section{Gluing and generalized arithmetic sequences}

Let $S_1=\langle n_1, \dots, n_{\nu} \rangle$ and $S_2=\langle m_1, \dots, m_{\mu} \rangle$ be two numerical semigroups and let $x \in S_2 \setminus \G(S_2)$, $y \in S_1 \setminus \G(S_1)$ be two coprime integers. The gluing of $S_1$ and $S_2$ with respect to $x$ and $y$ is the numerical semigroup $\langle xS_1, yS_2 \rangle=\langle xn_1, \dots, xn_{\nu},ym_1, \dots, ym_{\mu} \rangle$. It is well-known that $\langle xS_1, yS_2 \rangle$ is symmetric if and only if both $S_1$ and $S_2$ are symmetric and Nari \cite[Theorem 6.7]{N} proved that $\langle xS_1, yS_2 \rangle$ is never almost symmetric if it is not symmetric. In the next proposition we extend this result to the nearly Gorenstein case.

\begin{proposition} \label{gluing}
Let $S_1$ and $S_2$ be two numerical semigroups and assume that at least one of them is not symmetric. Then, every gluing of $S_1$ and $S_2$ is not nearly Gorenstein.
\end{proposition}

\begin{proof}
Assume that $S_2$ is not symmetric and let $m$ be the multiplicity of $S_1$, i.e. its smallest minimal generator. Suppose that $S=\langle xS_1, yS_2 \rangle$ is nearly Gorenstein with $x \in S_2 \setminus \G(S_2)$, $y \in S_1 \setminus \G(S_1)$ and $\gcd(x,y)=1$. It is well-known that $\PF(S)= \{xf +yf'+xy \mid f \in \PF(S_1), f' \in \PF(S_2)\}$, see \cite[Proposition 6.6]{N}. Then, by Proposition \ref{Nearly Gorenstein} there exists $xf +yf'+xy \in \PF(S)$ such that for every $g \in \PF(S_2)$ we have $xm + xf+yf'+xy -(xf+yg+xy)=xm+yf'-yg \in S$.  Since $S_2$ is not symmetric we can fix $g \in \PF(S_2) \setminus \{f'\}$.
Let $xm+ yf'-yg=xs_1+ys_2$ with $s_1 \in S_1$ and $s_2 \in S_2$. If $s_1=0$, then from $\gcd(x,y)=1$ it follows that $y$ divides $m$, but this impossible because $m < y$. Therefore, $y(f'-g-s_2)=x(s_1-m)\geq 0$ and, thus, $f'-g-s_2= \lambda x$ with $\lambda \in \mathbb{N}$. Since $x \in S_2$ and $g \in \PF(S_2)$, this implies that $f'=s_2+\lambda x +g \in S_2$, which is a contradiction.
\end{proof}

In literature there exists a construction that is a variation of the gluing when one semigroup is $\mathbb{N}$. More precisely, if $S=\langle n_1, \dots, n_\nu \rangle$ and $d \in \mathbb N$ such that $\gcd(d, n_\nu)=1$, we are interested in the semigroup $T=\langle dn_1, \dots, dn_{\nu-1}, n_\nu \rangle$. The numerical semigroup $T$ is symmetric if and only if $S$ is symmetric, see \cite[Proposition 8]{FGH}, whereas Numata \cite{Nu} proved that $T$ is never almost symmetric when $S$ is not symmetric. In the next proposition we show that this result holds also for the nearly Gorenstein property. We first recall that the pseudo-Frobenius numbers of $T$ are
\[
\PF(T)=\{df+(d-1)n_{\nu} \mid f \in \PF(S)\},
\]
in particular $\F(T)=d\F(S)+(d-1)n_{\nu}$, see \cite[Proposition 3.2]{Nu}.

\begin{proposition} \label{Numata}
Let $S=\langle n_1, \dots, n_\nu \rangle$ be a numerical semigroup which is not symmetric. If $d$ is a positive integer coprime to $n_\nu$, then $T=\langle dn_1, \dots, dn_{\nu-1},n_\nu \rangle$ is not nearly Gorenstein. 
\end{proposition}

\begin{proof}
Suppose by contradiction that $T$ is nearly Gorenstein. By Proposition \ref{Nearly Gorenstein} there exists $f \in \PF(S)$ such that $n_\nu +df-dg \in T$ for all $g \in \PF(S)$ and, since $S$ is not symmetric, we can fix $g \neq f$. Therefore, $n_\nu + df -dg = ds + \lambda n_\nu$ with $s \in S$ and $\lambda >0$, since $d$ does not divide $n_\nu$.
Then, $d(f-g-s)=(\lambda - 1) n_{\nu}$ and, since $\gcd(d, n_\nu)=1$, this implies that $f-g-s=\gamma n_\nu$ with $\gamma \geq 0$. Hence, $f=g+s+\gamma n_{\nu} \in S$ because $f \neq g$ and this yields a contradiction.
\end{proof}

The next result is a nice consequence of Propositions \ref{gluing} and \ref{Numata} and it was proved by Numata \cite{Nu} in the almost symmetric case.

\begin{corollary} \label{coprime}
Let $T$ be a nearly Gorenstein numerical semigroup which is not symmetric and assume that it is minimally generated by $n_1, \dots, n_\nu$. Then, the elements  of $\{n_1, \dots, \widehat{n_i}, \dots n_\nu \}$ are relatively coprime for every $i$.
\end{corollary}

\begin{proof}
It is enough to assume $i=\nu$. Suppose by contradiction that $\gcd(n_1, \dots, n_{\nu-1})=d >1$ and let $S=\langle n_1/d, \dots, n_{\nu-1}/d \rangle$. If $n_{\nu} \notin S$ we get a contradiction by Proposition \ref{Numata}. Hence, $S$ is the gluing $\langle dS, n_\nu \mathbb N \rangle$ and Proposition \ref{gluing} yields a contradiction.  
\end{proof}

A numerical semigroup generated by a generalized arithmetic sequence has the form $S=\langle a,sa+d, sa+2d, \dots, sa+nd \rangle$ for some positive integers $a,s,d,n$ such that $\gcd(a,d)=1$. It is known that in this case $S$ is symmetric if and only if $a \equiv 2 \mod n$, see \cite{EL,Ma}, while in \cite[Corollary 3.3]{Nu2} Numata proved that it is almost symmetric if and only if either $s=1$ and $S$ has maximal embedding dimension or it is symmetric. Moreover, in the arXiv version of \cite{HHS} it is proved that $S$ is always nearly Gorenstein provided that $s=1$.

\begin{proposition}
Let $S=\langle a,sa+d,\dots, sa+nd \rangle$ be a numerical semigroup generated by a generalized arithmetic sequence. It is nearly Gorenstein if and only if $s=1$ or $a \equiv 2 \mod n$.
\end{proposition}

\begin{proof}
We can exclude the symmetric case. By the proof of \cite[Lemma 2.7]{Ma} (see also \cite[Theorem 3.1]{Nu2}), $\F(S)-d$ is a pseudo-Frobenius number of $S$. If $S$ is nearly Gorenstein, then $f_1=\F(S)$ by Proposition \ref{minimum index} and, so, $a+\F(S)-(\F(S)-d)=a+d \in S$, that is possible only if $s=1$. Hence, the statement follows from Proposition 7.1 of the arXiv version of \cite{HHS}. 
\end{proof}

\section{Further questions and open problems}

In this section we collect some open problems. We start recalling the question raised by the first author in \cite{M} for the almost symmetric case and by Stamate \cite{S} in general.

\begin{question}
Is there an upper bound for the type of $S$ in terms of the embedding dimension of $S$ when $S$ is almost symmetric or nearly Gorenstein? 
\end{question}

To the best of our knowledge no almost symmetric semigroups $S$ for which $t(S)\geq 2\nu(S)$ are known, even though there are almost symmetric semigroups $S$ with embedding dimension $6$ satisfying $t(S)=2 \nu(S)-1$, for instance $S=\langle 111,115,122,126,135,146 \rangle$. 
Also, there exist nearly Gorenstein numerical semigroups which are not almost symmetric having embedding dimension $6$ and type $9$, cf. Remark \ref{type 9}. On the other hand some computations suggest that the inequality $t(S) \leq \nu(S)$ could hold if $\nu(S) < 6$. More precisely we pose the following question:

\begin{question}
Let $S$ be a nearly Gorenstein numerical semigroup with embedding dimension five. Is it true that $t(S)\leq 5$ and that the equality is attained only if $S$ is almost symmetric?
\end{question}

Let $S=\langle n_1, \dots, n_{\nu}\rangle$ and let $k$ be a field. The map $\varphi: k[x_1,\dots, x_\nu] \rightarrow k[S]$ defined as $\varphi(x_i)=t^{n_i}$ is surjective and its kernel $I_S$ is said to be the defining ideal of $S$. Clearly $k[S] \cong k[x_1,\dots, x_\nu]/I_S$.
The defining ideals of almost symmetric semigroups with embedding dimension $3$ or $4$ and type $1$ and $2$ are well-known, see \cite{B,H,RG,K} or the survey \cite{S}. Also in the case with embedding $4$ and type $3$ the defining ideal has been recently found in \cite{E,HW}. As for the nearly Gorenstein case, the defining ideal of $S$ is essentially described in the arXiv version of \cite{HHS} when $\nu(S)=3$. In \cite[Question 9.8]{S} Stamate asks for the generators and the resolution of the defining ideal $I_S$ when $\nu(S)=4$ and $S$ is nearly Gorenstein. We raise a more precise question on the number of its minimal generators. This is equivalent to ask for all the Betti numbers of $k[x_1,\dots,x_4]/I_S$ because its projective dimension is $3$.

\begin{question}
Let $S$ be a nearly Gorenstein numerical semigroup which is not almost symmetric and let $\nu(S)=4$. Are the following statements true?
\begin{enumerate}
\item If $t(S)=2$, then the defining ideal of $S$ has either $4$ or $5$ generators.
\item If $t(S)=3$, then the defining ideal of $S$ has $6$ generators.
\end{enumerate}
Equivalently, the possible Betti sequences of $k[x_1,\dots,x_4]/I_S$ are $(1,4,5,2)$, $(1,5,6,2)$ and $(1,6,8,3)$.
\end{question}

In \cite{R} the notion of ring with canonical reduction is introduced. More precisely, we say that a one-dimensional Cohen-Macaulay ring $(R,\mathfrak m)$ has a canonical reduction if there exists a canonical ideal $I$ of $R$ that is a reduction of $\mathfrak m$. In \cite[Theorem 3.13]{R} it is proved that a numerical semigroup ring $k[[S]]$ has a canonical reduction if and only if $n_1+\F(S)-g \in S$ for every $g \in \mathbb N \setminus S$. It is easy to see that this is equivalent to require that $n_1+\F(S)-f \in S$ for every $f \in \PF(S)$ and we say that $S$ has a canonical reduction if satisfies this property. Therefore, Proposition \ref{minimum index} implies that a nearly Gorenstein semigroup has a canonical reduction 
and, thus, we have the following implications:
\[
\text{ Almost symmetric} \Rightarrow \text{  Nearly Gorenstein} \Rightarrow \text{ Semigroup with canonical reduction.}
\]
It is natural to ask if Theorem \ref{t<4} is still true for numerical semigroups with four generators that have a canonical reduction. However, the semigroup $S=\langle 16,17,19,39 \rangle$ has type four and it is easy to see that it has a canonical reduction using the criterion above. 
On the other hand, several computations suggest the following question:

\begin{question}
Let $k[[S]]$ be a numerical semigroup ring with canonical reduction and assume that $S$ has embedding dimension four. Is $t(S) \leq 4$?
\end{question}

\medskip

\noindent
{\bf Acknowledgements}.
The authors would like to thank Alessio Sammartano for many useful discussions about the topics of this paper.

\end{document}